\documentclass[a4paper,12pt]{amsart}

\usepackage{fancyhdr}
\usepackage{a4wide}
\usepackage{tikz}
\usepackage{amsmath,amsthm}
\usepackage{float}
\usepackage{color}
\usepackage[noadjust]{cite}
\usepackage{enumitem}
\usepackage{cleveref}
\newcommand{\mylabel}[2]{#2\def\@currentlabel{#2}\label{#1}}
\DeclareMathOperator{\m}{m}

\newcommand{\A}{\alpha}

\DeclareMathOperator{\av}{avm}

\DeclareMathOperator{\Si}{S}

\DeclareMathOperator{\M}{M}
\newcommand{\suv}{\sum_{u: uv \in E(G)}}
\DeclareMathOperator{\ME}{ME}

\newcommand{\changefont}{%
    \fontsize{9}{12}\selectfont}

\newtheorem{thm}{Theorem}[section]
\newtheorem{prop}[thm]{Proposition}
\newtheorem{lem}[thm]{Lemma}
\newtheorem{cor}[thm]{Corollary}
\theoremstyle{remark}
\newtheorem{rem}[thm]{Remark}

\newtheorem*{thm*}{Theorem}

\newif\ifdetails
\detailstrue
\newcommand{\DETAIL}[1]%
{\ifdetails\par\fbox{\begin{minipage}{0.9\linewidth}\textit{Detail:}
      #1\end{minipage}}\par\fi}
\newcommand{\TODO}[1]%
{\ifdetails\par\fbox{\begin{minipage}{0.9\linewidth}\textbf{TODO:}
      #1\end{minipage}}\par\fi}

\newcommand{\old}[1]{{}}

\title{The average size of matchings in graphs}

\author{Eric O. D. Andriantiana}
\address{Eric O. D. Andriantiana\\
Department of Mathematics (Pure and Applied)\\
Rhodes University, PO Box 94\\
6140 Grahamstown\\
South Africa}
\email{E.Andriantiana@ru.ac.za}

\author{Valisoa Razanajatovo Misanantenaina}
\address{Valisoa Razanajatovo Misanantenaina\\
Department of Mathematical Sciences\\
Stellenbosch University\\
Private Bag X1\\
Matieland 7602\\
South Africa}
\email{valisoa@sun.ac.za}

\author{Stephan Wagner}
\address{Stephan Wagner\\
Department of Mathematical Sciences\\
Stellenbosch University\\
Private Bag X1\\
Matieland 7602\\
South Africa}
\email{swagner@sun.ac.za}

\thanks{This work was supported by the National Research Foundation of South Africa (grants 96236 and 96310).}

\subjclass[2010]{Primary 05C35; secondary 05C05, 05C07}
\keywords{Independent sets, average size, trees, extremal problems}

\begin{document}

\begin{abstract}
In this paper, we consider the average size of independent edge sets, also called matchings, in a graph. We characterize the extremal graphs for the average size of matchings in general graphs and trees. In addition, we obtain inequalities between the average size of matchings and the number of matchings as well as the matching energy, which is defined as the sum of the absolute values of the zeros of the matching polynomial.  
\end{abstract}

\maketitle

\section{Introduction}
An independent vertex set in a graph is a set of vertices such that no two vertices are adjacent. An independent edge set, also called a matching, is a set of edges such that no two edges are adjacent. It is not surprising that these two concepts are closely related, an elementary example being the fact that a matching in a graph is an independent set in the corresponding line graph. Two popular graph invariants associated to these parameters are the Merrifield-Simmons index and the Hosoya index, which are the total number of independent sets and the total number of matchings respectively. Extremal problems, where one is looking for the maximum or minimum of an invariant in a specified class of graphs, have been studied quite thoroughly for both the Merrifield-Simmons index and the Hosoya index. It is straightforward that among all $n$-vertex graphs, the complete graph has the maximum Hosoya index and the minimum Merrifield-Simmons index, while on the other hand the empty graph has the minimum Hosoya index and the maximum Merrifield-Simmons index. Among $n$-vertex trees, the path and the star are extremal, and there are numerous other examples of graph classes where the graphs that minimize the Merrifield-Simmons index also maximize the Hosoya index, and vice versa \cite{wagner2017upper}.

In a recent paper \cite{average}, we were interested in extremal questions for the average size of independent sets of graphs rather than their number. This was partly inspired by the work of Jamison \cite{jamison1983,jamison1984} and later authors \cite{vince2010,haslegrave2014,wagner2016,Mol} on the average size of subtrees of trees. In the present paper, which complements our paper \cite{average}, we are concerned with the study of the average size of matchings in a graph. In view of the aforementioned relation between independent sets and matchings, we expect to get similar results as for the average size of independent sets. Indeed, we find that the graphs that minimize the average size of independent sets are also those that maximize the average size of matchings and vice versa in all instances that we treat. Specifically, it holds true for arbitrary graphs and trees of a prescribed size.

Finally, we also prove inequalities between the average size of matchings and the number of matchings as well as the matching energy of a graph, an invariant introduced in \cite{gutman4}.

\section{Preliminaries}

Let $G$ be a graph. A subset $A$ of $E(G)$ is called a matching of $G$ if the edges of $A$ do not share any common vertices. Let $\m(G,k)$ be the number of matchings of cardinality $k$ (also called $k$-matchings) in $G$. We use the following notation for the total number of matchings in $G$, the sum of the sizes of all matchings in $G$ and the average size of matchings in $G$:
\begin{align*}
\M(G) &= \sum_{k\geq 0}\m(G,k),\\
\Si(G) &= \sum_{k\geq 0}k\m(G,k),\\
\av(G) &= \frac{\Si(G)}{\M(G)}.
\end{align*}

The greatest cardinality of a matching in $G$ is called the matching number of $G$ and denoted by $\mu(G)$.

As examples, let us consider the $n$-vertex edgeless graph $E_n$ and the star $S_n$. We have 
$$\M(E_n)=1,\quad \M(S_n)=n, \quad \Si(E_n)=0,\quad \Si(S_n)=n-1$$
and hence
$$
\av(E_n)=0, \quad \av(S_n)=\frac{n-1}{n}.
$$

The following standard and well-known proposition gives us a recursion for the total number and size of matchings.
\begin{prop}\label{prop:recursions}
If $e=uv$ is an edge of $G$, then
\begin{equation} \label{eq11match}
\M(G)=\M(G-e)+\M(G-v-u)
\end{equation}
and
\begin{equation}
\Si(G)=\Si(G-e)+ \Si(G-v-u)+ \M(G-v-u)\label{eq2match}.
\end{equation}

Similarly, if $v$ is a vertex of $G$, then
\begin{equation} \label{eq1match}
\M(G)=\M(G-v)+\suv \M(G-v-u)
\end{equation}
and
\begin{equation} \label{eq1match}
\Si(G)=\Si(G-v)+\suv (\Si(G-v-u)+\M(G-v-u)).
\end{equation}
\end{prop}

\begin{proof}
A matching in $G$ either contains the edge $e$ or not. The number of matchings containing $e$ is $\M(G-v-u)$, and the number of those not containing $e$ is $\M(G-e)$. Hence, the first equation holds. The argument for the second equation is similar, with the last term taking the edge $e$ itself into account.

Using similar reasoning, the last two equations are obtained by distinguishing between matchings that do not contain an edge with $v$ as an endpoint and those that do contain such an edge.
\end{proof}

\begin{rem}
In particular, if $v$ is a leaf of a tree and $w$ its unique neighbor, we obtain the relations
$$\M(G)=\M(G-v)+\M(G-v-w)$$
and
$$\Si(G) = \Si(G-v) + \Si(G-v-w) + \M(G-v-w).$$
\end{rem}

Moreover, we have the following basic result on disjoint unions:

\begin{prop}
Let $G_1,G_2,\ldots,G_k$ be the connected components of a graph $G$. Then we have
$$\M(G) = \prod_{j=1}^k \M(G_j)$$
and
$$\Si(G) = \sum_{i=1}^k \Si(G_i) \prod_{\substack{j=1 \\ j \neq i}}^k \M(G_j) = \M(G) \sum_{i=1}^k \frac{\Si(G_i)}{\M(G_i)},$$
thus
$$\av(G) = \sum_{i=1}^k \av(G_i).$$
\end{prop}

\begin{proof}
This follows easily from the fact that every matching of $G$ decomposes uniquely into matchings of its connected components.
\end{proof}

\section{General graphs}

Unlike the total number of matchings $\M$, the average size of matchings $\av$ is not always a monotone function under addition of edges to the graph. 
For example, consider the tree in Figure \ref{fig1}. We have
\begin{figure}[htbp]
$$
\begin{tikzpicture}[scale=1]
\node[fill=black,circle,inner sep=1pt] (t1) at (0,0) {};
\node[fill=black,circle,inner sep=1pt] (t2) at (1,0) {};
\node[fill=black,circle,inner sep=1pt] (t3) at (2,0) {};
\node[fill=black,circle,inner sep=1pt] (t4) at (3,0) {};
\draw (t1)--(t2)--(t3)--(t4);
\node[fill=black,circle,inner sep=1pt] (t5) at (1,1) {};
\draw (t5)--(t2);
\node at (1.5,0.2) {$e_1$};
\node at (2.5,0.2) {$e_2$};
\end{tikzpicture}
$$
\caption{A tree $T$ and two of its edges.}
\label{fig1}
\end{figure}
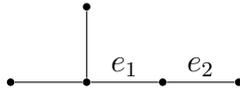

$$
\av(T-e_1)=\frac{7}{6}>\frac{8}{7}=\av(T),
\text{ but }
\av(T-e_2)=\frac{3}{4}<\frac{8}{7}=\av(T).
$$

However, we can make use of the following result obtained in \cite{average}:

\begin{thm}
\label{chap2.Thm:Gen}
Let $X$ be a nonempty finite set, and $\mathcal{P}(X)$ its powerset. For a set $\mathcal{A}\subseteq \mathcal{P}(X)$, 
we define
$$
av(\mathcal{A})=\frac{1}{|\mathcal{A}|}\sum_{A\in \mathcal{A}}|A|.
$$
Let $\mathcal{B}\subseteq \mathcal{P}(X)$, such that the cardinalities of the elements of $\mathcal{B}$ are not all the same and for every $x\in X$ there exists $B\in\mathcal{B}$ 
with $x\in B$. Then there exists $x_0\in X$ such that 
$$
av(\mathcal{B})>av\big(\mathcal{B}\cap \mathcal{P}(X-\{x_0\})\big).
$$
\end{thm}

Applying Theorem \ref{chap2.Thm:Gen}, with $\mathcal{B}$ being the set of matchings of $G$, we immediately obtain the following theorem.
\begin{thm}
If $G$ is a nonempty graph, then there exists an edge $e$ in $E(G)$ such that 
$$
\av(G-e)<\av(G).
$$
\end{thm}

As an immediate consequence, we have the following corollary (which of course is also rather trivial without Theorem~\ref{chap2.Thm:Gen}).

\begin{cor}
For every $n$-vertex graph $G$ that is not the edgeless graph $E_n$, 
$0=\av (E_n) < \av(G)$. \label{lem1}
\end{cor}

One might wonder whether there is an analogous statement for adding edges. If it was possible to add an edge to every non-complete graph in such a way that the average matching size increases, it would follow immediately that complete graphs maximize the invariant $\av$. While the latter is true (as will be shown in the following), the analogue of Theorem~\ref{chap2.Thm:Gen} fails, as the example of a four-vertex cycle shows: when an edge $e$ is added to the cycle $C_4$, we have
$$\av(C_4) = \frac87 > \frac98 = \av(C_4 + e).$$

Thus we need another approach to show that the complete graph is still extremal. For this purpose, we first introduce some notation. 

In analogy to $\M(G)$, $\Si(G)$ and $\av(G)$, we define the following partial quantities for every nonnegative integer $k$:
\begin{align*}
\M_k(G) &=\sum_{i=0}^k\m(G,i),\\
\Si_k(G) &=\sum_{i=0}^ki\m(G,i),\\
\av_k(G) &=\frac{S_k(G)}{\M_k(G)}.
\end{align*}

We have the following lemmas.

\begin{lem}
\label{Lem:0}
For every nonnegative integer $k$ and every graph $G$, we have 
$$
\av_{k+1}(G)\geq \av_k(G).
$$
If $k\geq \mu(G)$, then 
$$
\av_{k+1}(G)= \av_k(G)=\av(G).
$$
\end{lem}

\begin{proof}
This is straightforward from the definition of $\av_k$.
\end{proof}

\begin{lem}
\label{Lem:1}
For every $n$-vertex graph $G$ and every nonnegative integer $k$ such that $k < \mu(G)$, we have
$$
\frac{\m(K_n,k)}{\m(K_n,k+1)}\leq \frac{\m(G,k)}{\m(G,k+1)}.
$$
\end{lem}
\begin{proof}
Let $N$ be any $k$-matching of the complete graph $K_n$. When the $2k$ vertices that are covered by $N$ are removed, a complete graph on $n-2k$ vertices remains. Thus there are $\m(K_{n-2k},1) = \binom{n-2k}{2}$ possible ways to extend $N$ to a $(k+1)$-matching. Conversely, every $(k+1)$-matching can be obtained as an extension of $k+1$ different $k$-matchings. It follows that
$$\m(K_n,k+1) = \m(K_n,k) \cdot \frac{\m(K_{n-2k},1)}{k+1}.$$
Likewise, if $N$ is a $k$-matching of $G$ and $v(N)$ the set of vertices covered by $N$ in $G$, then there are $m(G - v(N))$ ways to extend $N$ to a $(k+1)$-matching of $G$. So by the same double-counting argument, we have
$$\m(G,k+1) = \frac{1}{k+1} \sum_{N:\,\text{$k$-matching of $G$}} \m(G - v(N),1).$$
Clearly, $\m(G - v(N),1) \leq \m(K_{n-2k},1)$ for all $k$-matchings $N$ (with equality if and only if $G - v(N)$ is complete), thus
$$\m(G,k+1) \leq \frac{1}{k+1} \cdot \m(G,k) \cdot \m(K_{n-2k},1),$$
and the desired inequality follows.
\end{proof}

\begin{rem}
Equality in Lemma~\ref{Lem:1} may hold for some (but not all) $k$ even if $G$ is not complete: for example, for the $4$-cycle $C_4$, we have
$$\frac{\m(K_4,2)}{\m(K_4,1)} = \frac36 = \frac24 = \frac{\m(C_4,2)}{\m(C_4,1)}.$$
\end{rem}
Lemma \ref{Lem:1} can easily be extended to the following lemma by induction:
\begin{lem}
\label{Lem:2}
For every $n$-vertex graph $G$ and for every pair of integers $k,l$ with $\mu(G)\geq k \geq l \geq 0$, we have 
$$
\frac{\m(K_n,l)}{\m(K_n,k)}\leq \frac{\m(G,l)}{\m(G,k)}
$$
and thus
$$
\frac{\M_{l}(K_n)}{m(K_n,k)} = \frac{\sum_{i=0}^l\m(K_n,i)}{\m(K_n,k)}\leq \frac{\sum_{i=0}^l\m(G,i)}{\m(G,k)} = \frac{\M_{l}(G)}{m(G,k)}.
$$
\end{lem}
\begin{thm}
\label{Th:1}
For every $n$-vertex graph $G$ and every integer $k$ with $\mu(G)\geq k > 0$, we have
$$
\av_k(K_n)\geq \av_k(G),
$$
with equality if and only if $G$ is a complete graph.
\end{thm}
\begin{proof}
We only need to consider the case that $G$ is not complete. Note first that
$$
\av_1(K_n)=\frac{|E(K_n)|}{|E(K_n)|+1} > \frac{|E(G)|}{|E(G)|+1}=\av_1(G).
$$
The inequality holds because $\frac{x}{x+1}$ is an increasing function of $x$ on the interval $[0,\infty)$.

Assume that $\av_k(K_n) > \av_k(G)$ for some positive integer $k$, $k <\mu(G)$. Then we have $\m(k+1,G)\neq 0$ and 
\begin{align}
\av_{k+1}(K_n)
&=\frac{(k+1)\m(K_n,k+1)+\sum_{i=0}^ki\m(K_n,i)}{\m(K_n,k+1)+\sum_{i=0}^k\m(K_n,i)}\nonumber\\
&=\frac{(k+1)\m(K_n,k+1)+\Si_k(K_n)}{\m(K_n,k+1)+\M_k(K_n)}\nonumber\\
&=\frac{(k+1)\m(K_n,k+1)+\av_k(K_n)\M_k(K_n)}{\m(K_n,k+1)+\M_k(K_n)}\nonumber\\
\label{Eq:1}
&=\frac{(k+1)+\av_k(K_n)\frac{\M_k(K_n)}{\m(K_n,k+1)}}{1+\frac{\M_k(K_n)}{\m(K_n,k+1)}}.
\end{align}
Since $k+1>\av_k(K_n)$, $\frac{(k+1) + \av_k(K_n)x}{1+x}$ is decreasing as a function of 
$x$ on the interval $[0,\infty)$, so Lemma \ref{Lem:2} and \eqref{Eq:1} imply that
\begin{align}
\av_{k+1}(K_n)
&\geq \frac{(k+1)+\av_k(K_n)\frac{\M_k(G)}{\m(G,k+1)}}{1+\frac{\M_k(G)}{\m(G,k+1)}}.
\end{align}
Finally, using the induction hypothesis $\av_k(K_n) > \av_k(G)$, we obtain
\begin{align}
\av_{k+1}(K_n)
&> \frac{(k+1)+\av_k(G)\frac{\M_k(G)}{\m(G,k+1)}}{1+\frac{\M_k(G)}{\m(G,k+1)}}=\av_{k+1}(G).
\end{align}

\end{proof}

\begin{cor}
For every $n$-vertex graph $G$ we have $\av(K_n)\geq \av(G)$, with equality only if $G$ is a complete graph.
\end{cor}
\begin{proof}
Theorem \ref{Th:1} and Lemma \ref{Lem:0} give us
$$
\av(K_n)=\av_{\lfloor n/2\rfloor}(K_n)\geq \av_{\mu(G)}(K_n)\geq \av_{\mu(G)}(G)=\av(G).
$$
\end{proof}

\begin{rem}
While there is no simple explicit formula for $\av(K_n)$, it can be expressed in terms of the number of matchings in complete graphs. Every edge of the complete graph $K_n$ is contained in $\M(K_{n-2})$ matchings, thus we have $\Si(K_n) = \binom{n}{2} \M(K_{n-2})$ and consequently
$$\av(K_n) = \frac{\Si(K_n)}{\M(K_n)} = \binom{n}{2} \frac{\M(K_{n-2})}{\M(K_n)}.$$
A relatively simple asymptotic formula can be provided as well. There is a straightforward bijection between matchings of $K_n$ and involutions of an $n$-element set (a permutation is called an involution if it is equal to its own inverse, or equivalently if all cycles are of length $1$ or $2$). Thus the number of matchings of $K_n$ is the same as the number of involutions of an $n$-element set, for which there is a well-known asymptotic formula (see \cite[Proposition VIII.2]{Flajolet}):
$$\M(K_n) \sim \frac{1}{\sqrt{2}} n^{n/2} e^{-n/2+\sqrt{n}-1/4}.$$
It follows that
$$\av(K_n) \sim \frac{n}{2}$$
as $n \to \infty$.
\end{rem}

\section{Trees}

In this section, we will be concerned with trees. Our main goal is to determine the maximum and minimum of $\av(T)$ when $T$ is a tree with $n$ vertices.
Let us first consider the problem of minimizing the average size of matchings. As it turns out, the minimum for trees is also the minimum for connected graphs in general.

\begin{thm}
For every connected $n$-vertex graph, $\av(S_n) \leq \av(G)$, with equality only if $G$ is a star.
\end{thm}

\begin{proof}
We have shown earlier that $\av(S_n)=\frac{n-1}{n}<1$. However any other connected graph $G$ (except for the complete graph $K_3$, for which $\av(K_3)=\frac34 > \frac23$) on $n$ vertices satisfies $\av(G)\geq 1$, since it possesses matchings of size greater than $1$, which make up for the empty set.

\end{proof}

The maximization problem requires more effort. Note that
the line graph of the $n$-vertex path $P_n$ is the $(n-1)$-vertex path $P_{n-1}$. This implies that the matchings of $P_{n}$ can be identified with the independent sets of $P_{n-1}$. Thus, the average size of matchings of $P_{n}$ is the same as the average size of the independent sets of $P_{n-1}$. A formula for this average size was determined in \cite{average}, where it was also shown that this average is in fact the minimum among trees of the same size.

\begin{lem}\label{lem:path}
The average size of matchings of the $n$-vertex path $P_n$ is
\begin{equation}\label{eq:tP_n}
\av(P_n)= \frac{5 - \sqrt{5}}{10} n + \frac{1-\sqrt{5}}{10} - \frac{n+1}{\sqrt{5}((-\phi^2)^{n+1}-1)},
\end{equation}
where $\phi = \frac{\sqrt{5}+1}{2}$ is the golden ratio. In particular,
\begin{enumerate}
\item[(a)] $\displaystyle \lim_{n \to \infty} \av(P_n) - \frac{5 - \sqrt{5}}{10} n = \frac{1-\sqrt{5}}{10},$
\item[(b)] $\displaystyle \av(P_n) \leq \frac{5-\sqrt{5}}{10} n + \frac{1}{\sqrt{5}} - \frac12$, with equality only for $n = 2$. For all positive integers $n \neq 2$, we even have $\displaystyle \av(P_n) \leq \frac{5-\sqrt{5}}{10} n + \frac{2}{\sqrt{5}} - 1$.
\end{enumerate}
\end{lem}

\begin{proof}
The formula for $\av(P_n)$ is taken from~\cite{average} (using the aforementioned correspondence between matchings of $P_n$ and independent sets of $P_{n-1}$). The limit in (a) is a straightforward consequence. For (b), one only needs to note that the sign of the final term in \eqref{eq:tP_n} alternates, and that its absolute value is decreasing in $n$ (see also~\cite{average}).
\end{proof}

For ease of notation, we set $a= \frac{5-\sqrt{5}}{10} \approx 0.27639320$ and $c_n = \av(P_n) - an$. Table \ref{tab1} gives values of $c_n$ for small $n$.

\begin{table}
\centering
\begin{tabular}{|c|c|c|c|}
\hline
$n$ & 1 & 2 & 3 \\
\hline
$c_n$ & $\frac{\sqrt{5}}{10} - \frac12 \approx -0.2764$ & $\frac{1}{\sqrt{5}} - \frac12 \approx -0.0528$ & $\frac{3}{2\sqrt{5}} - \frac56 \approx -0.1625$  \\
\hline
$n$  & 4 & 5 & 6 \\
\hline
$c_n$ & $\frac{2}{\sqrt{5}} - 1 \approx -0.1056$ & $\frac{\sqrt{5}}{2} - \frac{5}{4} \approx -0.1320$ & $\frac{3}{\sqrt{5}} - \frac{19}{13} \approx -0.1199$ \\
\hline
\end{tabular}
\vspace{0.5cm}

\caption{Values of $c_1,c_2,\ldots,c_5$.}
\label{tab1}
\end{table}

Before we prove the main result of this section, we require one more lemma:

\begin{lem}\label{lem:quot1}
For every tree $T$ and every vertex $v$ of $T$, we have
$$\frac{1}{1+ d(v)}\leq \frac{\M(T-v)}{\M(T)} \leq 1,$$
where $d(v)$ denotes the degree of $v$.
\end{lem}

\begin{proof}
Note first that $\M(T) = \M(T-v) + \suv \M(T-v-u)$. Since $T-v-u$ is a subgraph of $T-v$, we have $\M(T-v-u) \leq \M(T-v)$, hence $(1+d(v))\M(T-v) \geq \M(T)$, which proves the first inequality. The second inequality simply follows from the fact that $T-v$ is a subgraph of $T$, so matchings of $T - v$ are also matchings of $T$.
\end{proof}

\begin{thm}
For every tree $T$ of order $n$ that is not a path, we have the inequality $\av(T) \leq an + b$, where $b = (7\sqrt{5}-17)/10 \approx -0.13475241$. Consequently, the path maximizes the value of $\av(T)$ among all trees of order $n$.
\end{thm}

\begin{proof}
We prove the inequality by induction on $n$. For $n \leq 3$, there is nothing to prove since the only trees with three or fewer vertices are paths. Thus assume now that $n \geq 4$, and consider a vertex $v$ of the tree $T$ whose degree is at least $3$ (which must exist if $T$ is not a path). Denote the neighbors of $v$ by $v_1,v_2,\ldots,v_k$ and the components of $T - v$ by $T_1,T_2,\ldots,T_k$ (in such a way that $v_j$ is contained in $T_j$). Let $e$ be the edge between $v$ and $v_k$, and $T' = T - T_k$ be the tree obtained by removing $T_k$ from $T$. We have

\begin{align}
\av(T) &= \frac{\Si(T)}{\M(T)} = \frac{\Si(T-e) + \Si(T-v-v_k)+\M(T-v-v_k)}{\M(T)} \nonumber \\
&= \frac{\M(T-e)}{\M(T)} \cdot \frac{\Si(T-e)}{\M(T-e)} + \frac{\M(T-v-v_k)}{\M(T)} \cdot \Big(1 + \frac{\Si(T-v-v_k)}{\M(T-v-v_k)}\Big) \nonumber\\
&=\frac{\M(T-e)}{\M(T)}  \av(T-e) + \frac{\M(T) - \M(T-e)}{\M(T)} (1+ \av(T-v-v_k))\nonumber \\
&=\frac{\M(T-e)}{\M(T)}  (\av(T')+\av(T_k)) \label{eqpa} \\
&\qquad + \Big( 1 - \frac{\M(T-e)}{\M(T)} \Big) \Big(1 + \sum_{j=1}^{k-1} \av(T_j) + \av(T_k-v_k)\Big)\nonumber.
\end{align}
Set $A=\av(T')+\av(T_k)$ and $B=1 + \sum_{j=1}^{k-1} \av(T_j) + \av(T_k-v_k)$.

Assume first that $k \geq 4$. By Lemma~\ref{lem:path} and the induction hypothesis, we have $\av(T_j) \leq a|T_j| + \frac{1}{\sqrt{5}} - \frac12$ for all $j$ and $\av(T')\leq a|T'|+b$. It follows that
\[A \leq a(|T'|+|T_k|)+b+ \frac{1}{\sqrt{5}} - \frac12 = a |T| + b + \frac{1}{\sqrt{5}} - \frac12 < a|T|+b.\]

If $B \leq a|T|+b$, then we are done immediately. Hence we can assume that $A < a|T|+b \leq B$. This implies that the expression for $\av(T)$ in \eqref{eqpa} is decreasing regarded as a function of $\frac{\M(T-e)}{\M(T)}$, which means that we will need lower bounds for this quotient. So let us first find a formula for $\frac{\M(T-e)}{\M(T)}$. We observe that
$$\frac{\M(T-e)}{\M(T)}=\frac{\M(T')\M(T_k)}{\M(T')\M(T_k)+\M(T'-v)\M(T_k-v_k)},$$
thus
\begin{equation}\label{eq:Mte}
\frac{\M(T-e)}{\M(T)}=\Big(1+\frac{\M(T'-v)}{\M(T')}\cdot \frac{\M(T_k-v_k)}{\M(T_k)} \Big)^{-1}.
\end{equation}
Let us also find an expression for $\frac{\M(T'-v)}{\M(T')} $: 

\begin{align}
\frac{\M(T'-v)}{\M(T')}&=\frac{\prod_{j=1}^{k-1}\M(T_j)}{\prod_{j=1}^{k-1}\M(T_j)+\sum_{j=1}^{k-1} \M(T_j-v_j) \prod_{\substack{i=1 \\ i \neq j}}^{k-1} \M(T_i) }\nonumber\\
&= \Big(1+\sum_{j=1}^{k-1}\frac{\M(T_j-v_j)}{\M(T_j)}\Big)^{-1}.\label{eq:MT'}
\end{align}

We have to consider two different cases:

\medskip

\textbf{Case 1:} One of the $T_j$'s is the two-vertex path $P_2$. Then we can without loss of generality assume that $T_k=P_2$, so that $\av(T_k) = \frac12$ and $\av(T_k - v_k) = 0$. Let us distinguish two subcases depending on the number of other branches $T_j$ that are isomorphic to $P_2$.

\begin{itemize}

\item At least one of the $T_j$'s is different from $P_2$. We have
$$A \leq a|T|+b+\frac{1}{\sqrt{5}} - \frac12,$$
as it was established earlier. Moreover, by Lemma~\ref{lem:path} and the induction hypothesis,
$$\av(T_j) \leq a|T_j| + \frac{1}{\sqrt{5}} - \frac12$$
for all $j$, and
$$\av(T_j) \leq a|T_j| + \frac{2}{\sqrt{5}} - 1$$
if $T_j$ is different from $P_2$. Since this is the case for at least one index $j$, it follows that
\begin{align*}
B&= 1 + \sum_{j=1}^{k-1} \av(T_j) \\
&\leq 1+ \sum_{j=1}^{k-1} a|T_j| + (k-2) \left(\frac{1}{\sqrt{5}}-\frac12\right)+\frac{2}{\sqrt{5}}-1\\
&= a (|T| - 3) + (k-2) \left(\frac{1}{\sqrt{5}}-\frac12\right)+\frac{2}{\sqrt{5}} \\
&\leq a(|T|-3) + 2 \left(\frac{1}{\sqrt{5}}-\frac12\right)+\frac{2}{\sqrt{5}}\\
&= a|T|-3a+\frac{4}{\sqrt{5}}-1.
\end{align*}

Moreover, from the inequality $\frac{\M(T'-v)}{\M(T')} \leq 1 $ (see Lemma \ref{lem:quot1}) and the fact that $\frac{\M(T_k-v_k)}{\M(T_k)}= \frac{1}{2}$, using Equation \eqref{eq:Mte}, we obtain $\frac{\M(T-e)}{\M(T)} \geq \frac23$. Hence, \eqref{eqpa} gives us
\begin{align*}
\av(T)&\leq a|T| + \frac23\left(b+\frac{1}{\sqrt{5}} - \frac12\right)+\frac13\left(-3a+\frac{4}{\sqrt{5}}-1\right)\\
&= a|T|+\frac{29}{6\sqrt{5}}-\frac{23}{10}\approx a|T|-0.13847 < a|T|+b.
\end{align*}

\item All of the $T_j$'s are equal to $P_2$. In this case, we can determine $\M(T)$ and $\Si(T)$ explicitly (as functions of $k$ only) by means of Proposition~\ref{prop:recursions}:
$$\M(T) = 2^k + k2^{k-1}$$
and
$$\Si(T) = k2^{k-1} + k(k+1)2^{k-2},$$
thus
$$\av(T) = \frac{k^2+3k}{2k+4}.$$
Now one verifies easily that
$$\av(T) = \frac{k^2+3k}{2k+4} \leq a(2k+1)+b = a|T| + b$$
holds for all $k \geq 4$, completing the proof in Case 1.

\end{itemize}

\textbf{Case 2:} None of the $T_j$'s is a $2$-vertex path $P_2$.

By Lemma \ref{lem:quot1}, we have $\frac{\M(T'-v)}{\M(T')} \leq 1 $, and plugging this estimate into Equation \eqref{eq:Mte}, we obtain 
\begin{equation}\label{eq:Mt1}
\frac{\M(T-e)}{\M(T)}\geq \Big(1+\frac{\M(T_k-v_k)}{\M(T_k)} \Big)^{-1}.
\end{equation}

Let us distinguish different cases depending on the shape of $T_k$. We may assume that $T_k$ is the smallest branch, i.e.~$|T_k|=\min_{1\leq j\leq k}|T_j|$.

\begin{itemize}
\item If $|T_k|=1$, then $\av(T_k)=\av(T_k-v_k)=0$. It follows that
$$A = \av(T') \leq a|T'|+b=a|T|+b-a.$$
Moreover, since $\av(T_j) \leq |T_j| + \frac{2}{\sqrt{5}} - 1$ for every $j$ by the induction hypothesis and Lemma~\ref{lem:path} (and the assumption that none of the $T_j$ is a $2$-vertex path), we have
\begin{align*}
B &= 1 + \sum_{j=1}^{k-1} \av(T_j) \\
&\leq 1 + a \sum_{j=1}^{k-1} |T_j| + (k-1) \left(\frac{2}{\sqrt{5}}-1\right) \\
&\leq 1 + a (|T|-2) + 3 \left(\frac{2}{\sqrt{5}}-1\right) \\
&= a|T|-2a+\frac{6}{\sqrt{5}}-2.
\end{align*}
Since $\frac{\M(T_k-v_k)}{\M(T_k)}=1$, Equation \eqref{eq:Mt1} gives us $\frac{\M(T-e)}{\M(T)} \geq \frac12$. Thus,
\begin{align*}
\av(T)&\leq \frac12\left(a|T|+b-a\right)+\frac12\left(a|T|-2a+\frac{6}{\sqrt{5}}-2\right)\\
&= a|T|+\frac{11}{2\sqrt{5}}-\frac{13}{5}\approx a|T| -0.14033 < a|T|+b.
\end{align*}

\item If $|T_k|=3$, then $\av(T_k)=a|T_k|+\frac{3}{2\sqrt{5}}-\frac56$ and $\av(T_k-v_k)\leq a(|T_k|-1)+\frac{1}{\sqrt{5}}-\frac12$. In the same way as in the previous case, it follows that
\begin{align*}
A &\leq a|T|+b+\frac{3}{2\sqrt{5}}-\frac56,\\
B&\leq 1+a(|T|-2)+3\left(\frac{2}{\sqrt{5}}-1\right) + \frac{1}{\sqrt{5}}-\frac12= a|T|-2a+\frac{7}{\sqrt{5}}-\frac52.
\end{align*}
Since $\frac{\M(T_k-v_k)}{\M(T_k)} \leq \frac{2}{3}$, in this case Equation \eqref{eq:Mt1} gives us $\frac{\M(T-e)}{\M(T)} \geq \frac35$. We obtain
\begin{align*}
\av(T)&\leq a|T|+\frac35\left(b+\frac{3}{2\sqrt{5}}-\frac56\right)+\frac25\left(-2a+\frac{7}{\sqrt{5}}-\frac52\right)\\
&= a|T|+\frac{31}{5\sqrt{5}}-\frac{73}{25}\approx a|T|-0.14728 < a|T|+b.
\end{align*}

\item If $|T_k|=4$, then $\av(T_k)\leq a|T_k|+\frac{2}{\sqrt{5}}-1$ and $\av(T_k-v_k)\leq a(|T_k|-1)+\frac{3}{2\sqrt{5}}-\frac56$. In the same way as before, it follows that
\begin{align*}
A &\leq a|T|+b+\frac{2}{\sqrt{5}}-1,\\
B&\leq 1+a(|T|-2)+3\left(\frac{2}{\sqrt{5}}-1\right)+\frac{3}{2\sqrt{5}}-\frac56= a|T|-2a+\frac{3\sqrt{5}}{2}-\frac{17}{6}.
\end{align*}
Moreover, $\frac{\M(T_k-v_k)}{\M(T_k)} \leq \frac{3}{4}$ in this case, so using Equation \eqref{eq:Mt1} again, we get $\frac{\M(T-e)}{\M(T)} \geq \frac47$. Hence
\begin{align*}
\av(T)&\leq a|T|+\frac47\left(b+\frac{2}{\sqrt{5}}-1\right)+\frac37\left(-2a+\frac{15}{2\sqrt{5}}-\frac{17}{6}\right)\\
&= a|T|+\frac{19\sqrt{5}}{14}-\frac{223}{70}\approx a|T|-0.15105 < a|T|+b.
\end{align*}

\item If $|T_k|\geq 5$, then $\av(T_j)\leq a|T_j|+\frac{3}{\sqrt{5}}-\frac{19}{13}$ for all $j$ (by the induction hypothesis and Lemma~\ref{lem:path}, we have $\av(T_j) \leq a|T_j| + c_6 = a|T_j| + \frac{3}{\sqrt{5}}-\frac{19}{13}$ if $T_j$ is a path, and $\av(T_j) \leq a|T_j| + b \leq a|T_j| + \frac{3}{\sqrt{5}}-\frac{19}{13}$ otherwise) and $\av(T_k-v_k)\leq a(|T_k|-1)+\frac{2}{\sqrt{5}}-1$. So it follows now that
\begin{align*}
A &\leq a|T|+b+\frac{3}{\sqrt{5}}-\frac{19}{13},\\
B&\leq 1+a(|T|-2)+3\left(\frac{3}{\sqrt{5}}-\frac{19}{13}\right) +\frac{2}{\sqrt{5}}-1 = a|T|-2a+ \frac{11}{\sqrt{5}}-\frac{57}{13}.
\end{align*}

Since $\frac{\M(T_k-v_k)}{\M(T_k)}\leq 1$, we have $\frac{\M(T-e)}{\M(T)} \geq \frac12$ by~\eqref{eq:Mt1}. Thus,
\begin{align*}
\av(T)&\leq a|T| +\frac12 \left(b+\frac{3}{\sqrt{5}}-\frac{19}{13} -2a+ \frac{11}{\sqrt{5}}-\frac{57}{13} \right)\\
&= a|T|+\frac{37}{4\sqrt{5}}-\frac{1111}{260}\approx a|T|-0.13635 < a|T|+b.
\end{align*}

\end{itemize}

This completes the proof in the case that $k \geq 4$, so we are left with the case $k=3$. We return to the representation
\begin{align}\label{eq:mrep}
\av(T) &= \frac{\M(T-e)}{\M(T)} (\av(T')+\av(T_k)) \\
&\qquad + \Big(1-\frac{\M(T-e)}{\M(T)}\Big) \Big(1 + \sum_{j=1}^{k-1} \av(T_j) +\av(T_k-v_k)\Big)\nonumber. 
\end{align}

Plugging \eqref{eq:MT'} into Equation \eqref{eq:Mte}, we obtain
\begin{equation}\label{Mt}
\frac{\M(T-e)}{\M(T)}=\Big(1+\frac{1}{1+\sum_{j=1}^{k-1}\frac{\M(T_j-v_j)}{\M(T_j)}}\cdot \frac{\M(T_k-v_k)}{\M(T_k)} \Big)^{-1}.
\end{equation}

Now we distinguish different cases depending on how many of the branches $T_j$ have one, two, three, four and five or more vertices respectively. This gives us a total of $35$ cases corresponding to the solutions of
$$x_1+x_2+x_3+x_4+x_5 = 3.$$
Here, $x_1,x_2,x_3,x_4$ stand for the number of $T_j$'s with one, two, three, and four vertices respectively, and $x_5$ is the number of $T_j$'s with five or more vertices. In each of the cases, we use the following explicit values and bounds. The bounds and explicit values for $|T_j| \leq 4$ are obtained by an exhaustive case check, while the bounds for $|T_j| > 4$ follow from the induction hypothesis and Lemma~\ref{lem:path}.
$$\av(T_j) \begin{cases}
= a|T_j|-a &|T_j| = 1, \\
= a|T_j|+c_2 &|T_j| = 2, \\
= a|T_j|+c_3 &|T_j| = 3, \\
\leq a|T_j|+c_4&|T_j|=4,\\
\leq a|T_j| + c_6 & \text{otherwise,}
\end{cases}$$
$$\av(T_j-v_j) \begin{cases}
= a|T_j|-a &|T_j| = 1, \\
= a|T_j|-2a &|T_j| = 2, \\
\leq a(|T_j|-1)+c_2 &|T_j| = 3,\\
\leq a(|T_j|-1)+c_3 &|T_j| = 4,\\
\leq a(|T_j|-1)+c_4 & \text{otherwise.}
\end{cases}$$

We can assume that the degree of $v_j$ is at most $3$ for every $j$, since otherwise we can go back to the case that $k \geq 4$. Using this assumption, we have
$$\frac{\M(T_j-v_j)}{\M(T_j)} \begin{cases}
= 1 &|T_j| = 1, \\
= \frac12 &|T_j| = 2, \\
\in [\frac13,\frac23] &|T_j| = 3, \\
\in [\frac25, \frac34] &|T_j| = 4, \\
\in [\frac{4}{11},\frac34] & \text{otherwise.}
\end{cases}$$
The first four statements are obtained by checking all possible cases. For the last one, we use the recursion in~\eqref{eq:MT'} combined with Lemma \ref{lem:quot1}. Note first that $v_j$ has at most two neighbors in $T_j$, since its degree in $T$ is at most $3$. If there is only one neighbor, let $w$ be this neighbor, and set $S = T_j - v_j$. We have
$$\frac{\M(T_j-v_j)}{\M(T_j)} = \Big(1 + \frac{\M(S-w)}{\M(S)} \Big)^{-1}.$$
Applying Lemma~\ref{lem:quot1} to $S$ and $w$ yields $\frac13 \leq \frac{\M(S-w)}{\M(S)} \leq 1$ (if the degree of $w$ was greater than $2$, we could go back to the case $k \geq 4$ again), thus $\frac{\M(T_j-v_j)}{\M(T_j)} \in [\frac12,\frac34]$. If there are two neighbors $w_1$ and $w_2$, let $S_1$ and $S_2$ be the respective components of $T_j - v_j$. Since $\frac13 \leq \frac{\M(S_i-w_i)}{\M(S_i)} \leq 1$, we obtain
$$\frac{\M(T_j-v_j)}{\M(T_j)} = \Big(1 + \frac{\M(S_1-w_1)}{\M(S_1)}+\frac{\M(S_2-w_2)}{\M(S_2)}\Big)^{-1} \leq \frac{1}{1+\frac13 + \frac13} = \frac35$$
in this case, which readily proves the upper bound of $\frac34$ in all cases. To improve the lower bound even further, we can note that one of the two trees $S_1$ and $S_2$ has more than one vertex; without loss of generality, let this be $S_1$. Applying the same argument to $S_1$ as to $T_j$, we find
$\frac{\M(S_1-w_1)}{\M(S_1)} \leq \frac34$. Thus
$$\frac{\M(T_j-v_j)}{\M(T_j)} = \Big(1 + \frac{\M(S_1-w_1)}{\M(S_1)}+\frac{\M(S_2-w_2)}{\M(S_2)} \Big)^{-1} \geq \frac{1}{1+\frac34 + 1} = \frac{4}{11},$$
and we have also established the lower bound.

\medskip

Next we return to the representation~\eqref{eq:mrep}. By the induction hypothesis and Lemma~\ref{lem:path}, we have $\av(T')+\av(T_k)\leq (a|T'|+c_4)+(a|T_k|+c_2)<a|T|+b$. As before, if $1+\sum_{j=1}^{k-1}\av(T_j) + \av(T_j-v_j)\leq a|T|+b$, then we are done. So we may assume that 
$$\av(T')+\av(T_k)<a|T|+b\leq 1+\sum_{j=1}^{k-1}\av(T_j)+ \av(T_j-v_j).$$
Hence the expression~\eqref{eq:mrep} is linear and decreasing in $\frac{\M(T-e)}{\M(T)}$, its maximum is attained for the smallest possible value of $\frac{\M(T-e)}{\M(T)}$. 

\medskip

By the induction hypothesis, $\av(T')\leq \av(P_{|T'|}) = a|T'| + c_{|T'|}$. This inequality is plugged into~\eqref{eq:mrep} along with the bounds for $\av(T_j)$ and $\av(T_j-v_j)$. The identity~\eqref{Mt} is used to obtain a lower bound on the quotient $\frac{\M(T-e)}{\M(T)}$. All this gives us an upper bound for $\av(T)$ in each of the aforementioned $35$ cases, which can all be checked easily with a computer. The worst case happens when $x_1=x_3=x_4=x_5=0$ and $x_2=3$, where we have the equality $\av(T)=a|T|+b$. As another example to illustrate the general procedure, let us consider the case that gives us the second worst estimate: it is obtained for $x_1=x_3=x_5=0$, $x_2=2$ and $x_4=1$. Let $T_2$ and $T_3$ both have two vertices, so that the first branch $T_1$ consists of four vertices. We have
$$\av(T_3) = a|T_3|+ c_2,\ \av(T') \leq a|T'|+c_7,$$
thus
$$\av(T') + \av(T_3) \leq a|T| + c_2 + c_7 = a|T| + \frac{9}{2\sqrt{5}} - \frac{46}{21}.$$
Moreover, 
$$\av(T_3-v_3) = 0,\ \av(T_1) \leq 4a + c_4,\ \av(T_2) = \frac12,$$
and thus
\begin{align*}
1 + \sum_{j=1}^2 \av(T_j) + \av(T_3-v_3) \leq 1 + 4a+c_4 + \frac12 = a|T| + \frac{9}{2\sqrt{5}} - 2.
\end{align*}
Finally, we have
$$\frac{\M(T-e)}{\M(T)} = \Big(1 + \frac{\M(T_3-v_3)}{\M(T_3)} \cdot \frac{1}{1+\frac{\M(T_1-v_1)}{\M(T_1)}+\frac{\M(T_2-v_2)}{\M(T_2)}}\Big)^{-1} \geq
\Big(1+\frac12 \cdot \frac{1}{1+\frac12 + \frac25}\Big)^{-1} = \frac{19}{24}.$$
Putting everything together, we obtain
\begin{align*}
\av(T) &= \frac{\M(T-e)}{\M(T)} \big( \av(T') + \av(T_3) \big) \\
&\qquad + \left(1-\frac{\M(T-e)}{\M(T)}\right) \Big( 1 + \sum_{j=1}^2 \av(T_j) + \av(T_3-v_3)\Big) \\
&\leq \frac{\M(T-e)}{\M(T)} \Big( a|T| + \frac{9}{2\sqrt{5}} - \frac{46}{21} \Big) + \left(1-\frac{\M(T-e)}{\M(T)}\right) \Big(a|T| + \frac{9}{2\sqrt{5}} - 2 \Big) \\
&\leq \frac{19}{24} \Big( a|T| + \frac{9}{2\sqrt{5}} - \frac{46}{21} \Big) + \frac{5}{24} \Big(a|T| + \frac{9}{2\sqrt{5}} - 2 \Big) \\
&= a|T| + \frac{9}{2\sqrt{5}}-\frac{271}{126}\approx a|T| -0.13833 <a|T|+b.
\end{align*}
The other cases are treated in the same fashion and give upper bounds with smaller constant terms. Thus the induction is complete.
In order to complete the proof of the theorem, it only remains to prove an upper bound on $\av(P_n)$. However, we already know from Lemma~\ref{lem:path} that
\begin{align*}\av(P_n)  &= an +\frac{1-\sqrt{5}}{10} - \frac{n+1}{\sqrt{5}((-\phi^2)^{n+1}-1)} \\&\geq an +\frac{1-\sqrt{5}}{10} - \frac{6}{\sqrt{5}((-\phi^2)^{6}-1)} = an + \frac{\sqrt{5}}{2} - \frac{5}{4}
\end{align*}
for $n > 3$, and $\frac{\sqrt{5}}{2} - \frac{5}{4} \approx -0.131966 > b$. Thus $\av(P_n) > an + b \geq \av(T)$ for every tree $T$ with $n$ vertices other than $P_n$. This completes the proof.
\end{proof}

\section{Relations to other invariants}

In this section, we will prove inequalities between the average matching size and other matching-related quantities associated with a graph. 
Let $G$ be an $n$-vertex graph. The matching polynomial and the matching generating polynomial are defined as follows:
\begin{align*}
&\Phi(G,x)=\sum_{k\geq 0}\m(G,k)(-1)^k x^{n-2k},\\
&\M(G,x)=\sum_{k\geq 0}\m(G,k)x^k.
\end{align*}
Note that the average size of matchings in $G$ can be expressed as
\[\av(G)=\frac{\sum_{k\geq 0}k\m(G,k)}{\sum_{k\geq 0}\m(G,k)}=\frac{\M'(G,1)}{\M(G,1)},\]
where $\M'(G,x)$ is the first derivative of $\M(G,x)$ with respect to $x$.

It is easy to see that $\Phi(G,x) = x^n \M(G,-\frac{1}{x^2})$.  Using this relation, we can write the derivative of $\Phi$ in terms of $M$ and its derivative.

\begin{align*}
\Phi'(G,x)=nx^{n-1}\M\left(G,-\frac{1}{x^2}\right)+2x^{n-3}\M'\left(G,-\frac{1}{x^2}\right).
\end{align*}

This gives us
\begin{equation} \label{eq1}
\frac{\Phi'(G,x)}{\Phi(G,x)}=\frac{n}{x}+\frac{2}{x^3}\frac{\M'\left(G,-\frac{1}{x^2}\right)}{\M\left(G,-\frac{1}{x^2}\right)}.
\end{equation}

Let $\mu_1,\mu_2,\dots,\mu_n$ be the zeros of the matching polynomial $\Phi(G,x)$; it is well known that these zeros are real, see for example Section 8.5 in \cite{LovaszPlummer}. Now we can express $\Phi$ and $\Phi'$ in terms of the zeros as follows:

\begin{align*}
\Phi(G,x)&=\prod_{j=1}^n(x-\mu_j),\\
\Phi'(G,x)&=\sum_{k=1}^n \frac{\prod_{j=1}^n(x-\mu_j)}{x-\mu_k}.
\end{align*}

Therefore,

\begin{equation}\label{eq2}
\frac{\Phi'(G,x)}{\Phi(G,x)}=\sum_{k=1}^n\frac{1}{x-\mu_k}.
\end{equation}

Now, we can establish a relation between the average size of matchings of $G$ and the zeros of its matching polynomial.

\begin{lem}\label{lem1ww}
Let $G$ be an $n$-vertex graph and $\mu_1,\dots,\mu_n$ be the zeros of the matching polynomial of $G$. Then
\[\av(G)=\frac12\sum_{j=1}^n\frac{\mu_j^2}{\mu_j^2+1}.\]
\end{lem}

\begin{proof}
Using \eqref{eq1} and \eqref{eq2}, and plugging in $x=i$, we obtain
$$\sum_{j=1}^n\frac{1}{i-\mu_j} = \frac{n}{i}+\frac{2}{i^3}\frac{\M'\left(G,1\right)}{\M\left(G,1\right)},$$
and this simplifies to
\begin{equation}\label{eq3}
\sum_{j=1}^n\frac{\mu_j}{\mu_j-i} = 2\av(G).
\end{equation}

Let us rearrange the left hand side of Equation \eqref{eq3}. We have
\[\sum_{j=1}^n\frac{\mu_j}{\mu_j-i}=\sum_{j=1}^n\frac{\mu_j(\mu_j +i)}{(\mu_j-i)(\mu_j+i)}=\sum_{j=1}^n\frac{\mu_j^2+i\mu_j}{\mu_j^2+1}.\]

Since the imaginary part must be $0$, we get the desired result.
\end{proof}

Having established this relation, we can now prove two inequalities. The first relates the average matching size with the total number of matchings. Note that the latter is $\M(G) = \M(G,1)$, which can be expressed in terms of the zeros $\mu_1,\dots,\mu_n$ as well:
$$\M(G) = \M(G,1) = |\M(G,1)| = |i^{-n} \Phi(G,i)| = \Big| \prod_{j=1}^n (i-\mu_j) \Big| = \prod_{j=1}^n \sqrt{1+\mu_j^2}.$$
It is not difficult to verify that the inequality
$$\frac{x}{1+x} \leq \beta \log(1+x) + 1 - \beta + \beta \log \beta$$
holds for all positive real numbers $\beta$ and $x$. Plugging in $\mu_j^2$ for $x$ and summing over all $j$ yields the following result:

\begin{prop}
For every postive real number $\beta$ and every $n$-vertex graph $G$, we have
$$\av(G) \leq \beta \log \M(G) + \big( 1-\beta+\beta \log \beta \big) \frac{n}{2}.$$
In particular,
$$\av(G) \leq \log \M(G).$$
\end{prop}

We can still choose $\beta$ arbitrarily. Differentiating with respect to $\beta$, we find that the optimal value for $\beta$ (that minimizes the upper bound) is $\beta = \M(G)^{-2/n}$. Plugging this back into the inequality, we obtain the following theorem:

\begin{thm}
For every $n$-vertex graph, we have
$$\av(G) \leq \frac{n}{2} \Big( 1 - \M(G)^{-2/n} \Big).$$
\end{thm}

An alternative way to prove this theorem is to apply the inequality between the arithmetic and the geometric mean.

We conclude this section with a similar inequality involving the matching energy. This invariant is defined as follows \cite{gutman4}:

\[\ME(G)=\sum_{j=1}^n|\mu_j|.\]

Following an analogous approach, we can prove a relation between the average size of matchings in $G$ and the matching energy of $G$.

\begin{thm}\label{thm:MEineq}
For every graph $G$,
\[\ME(G)\geq 4\av(G).\]
\end{thm}

\begin{proof}
For all nonnegative real $x$, we have $\frac{x^2}{1+x^2} \leq \frac{x}{2}$. Therefore, by Lemma~\ref{lem1ww},
$$\av(G) = \frac12 \sum_{j=1}^n \frac{\mu_j^2}{1+\mu_j^2} \leq \frac12 \sum_{j=1}^n \frac{|\mu_j|}{2} = \frac14 \ME(G).$$
\end{proof}

\begin{rem}
Note that in the case of trees, the matching polynomial coincides with the characteristic polynomial. So we have a correspondence between the average size of matchings of a tree and the classical energy of a tree, which is the sum of the absolute values of the eigenvalues, see \cite{li2012}.
\end{rem}

\section{The weighted average size of matchings in a graph}

In the context of the monomer-dimer model from statistical physics, one often considers a probability distribution on the set of matchings where the probability of a $k$-matching is proportional to $\A^k$ for some constant $\A$, see for example \cite{davies2015}. This provides the motivation to study the weighted average size of matchings. We consider a random matching according to the aforementioned probability distribution, where $\A$ is a fixed positive number. We define the weighted total number of matchings in $G$, the weighted total size of $G$ and the weighted average size of matchings in $G$ as follows:
\begin{align*}
&\M^{\A}(G)=\sum_{k\geq 0}\m(G,k)\A^k,\\
&\Si^{\A}(G)=\sum_{k\geq 0}k\m(G,k)\A^k,\\
&\av^{\A}(G)=\frac{\Si^{\A}(G)}{\M^{\A}(G)}.
\end{align*}

Following a similar reasoning as in the special case where $\A = 1$, it is still possible to prove the following inequalities.
\begin{thm} For every fixed positive real number $\A$ and every $n$-vertex graph $G$, we have
$$\av^{\A} (E_n) \leq \av^{\A}(G)\leq \av^{\A}(K_n).$$ 
Moreover, for every real number $\A \in (0,1]$ and every $n$-vertex tree $T$, we have $$\av^{\A}(S_n) \leq \av^{\A}(T) \leq \av^{\A}(P_n).$$
\end{thm}

We refer to \cite{Valisoa} for more details on the proof.  Note that the final inequality ($\av^{\A}(T) \leq \av^{\A}(P_n)$) is not generally true for all values of $\alpha$. One can also express the weighted average matching size in terms of the zeros of the matching polynomial:
\begin{lem}
Let $G$ be an $n$-vertex graph and $\mu_1,\dots,\mu_n$ be the zeros of the matching polynomial of $G$. Then
\[\av^{\A}(G)=\frac12\sum_{j=1}^n\frac{\A\mu_j^2}{\A\mu_j^2+1}.\]
\end{lem}

Finally, it is also possible again to prove inequalities that relate $\av^{\A}(G)$ to other invariants. Specifically, we have the following straightforward generalization of Theorem~\ref{thm:MEineq}:
\begin{thm}
For every graph $G$ and every positive real number $\A$,
\[\ME(G)\geq \frac{4}{\sqrt{\A}}\av^{\A}(G).\]
\end{thm}

\bibliographystyle{abbrv}

\end{document}